\documentclass[reqno,11pt]{amsart}

\usepackage{lineno,amsmath,amsfonts,amsthm,amssymb,graphics,epsfig,mathrsfs, eufrak}
\usepackage{enumerate,geometry,etoolbox, scalerel, stackengine, relsize}
\usepackage{amscd,latexsym,url,upgreek,mathtools,color}
\usepackage{mathrsfs}
\usepackage{hyperref}
\usepackage[english]{babel}

\usepackage[utf8]{inputenc}
\usepackage[T1]{fontenc}

\usepackage{tikz}
\usetikzlibrary{arrows.meta}
\usepackage{booktabs}

\stackMath
\newcommand\reallywidehat[1]{%
\savestack{\tmpbox}{\stretchto{%
  \scaleto{%
    \scalerel*[\widthof{\ensuremath{#1}}]{\kern-.6pt\bigwedge\kern-.6pt}%
    {\rule[-\textheight/2]{1ex}{\textheight}}
  }{\textheight}%
}{0.5ex}}%
\stackon[1pt]{#1}{\tmpbox}%
}

\newtheorem{theorem}{Theorem}[section]
\newtheorem{lemma}[theorem]{Lemma}

\theoremstyle{definition}
\newtheorem{remark}[theorem]{Remark}
\newtheorem{definition}[theorem]{Definition}

\title{Approximation of Fractals via Lagrange-type Superoscillations}

\author[F. Mantovani]{F. Mantovani}
\address{(FM) Politecnico di
Milano\\Dipartimento di Matematica\\Via E. Bonardi, 9\\20133 Milano\\Italy}
\email{francesco.mantovani@polimi.it}

\author[D. C. Struppa]{D. C. Struppa}
\address{(DCS) The Donald Bren Presidential Chair in Mathematics\\ Chapman University, Orange, CA 92866 \\ USA}
\email{struppa@chapman.edu}

\begin{document}

\maketitle

\begin{abstract}
We study the approximation of the Weierstrass  function
by means of superoscillating sequences.
Superoscillatory functions are band-limited functions whose local oscillation rate
can exceed the highest frequency contained in their Fourier spectrum.
Starting from Lagrange-type interpolation at nodes in $[-1,1]$,
we construct a double-indexed family $\mathcal{W}_{N,n}(x)$
that approximates the truncated Weierstrass function $W_N(x)$ for each
fixed truncation order~$N$.
We prove that if the number of interpolation nodes $n_N$ grows sufficiently fast
relative to the highest frequency $b^N\pi$,
namely $b^N\pi/n_N\to 0$,
then $\mathcal{W}_{N,n_N}$ converges uniformly to the full
Weierstrass function on every compact set.
We also show that the two limits in $N$ and $n$
do \emph{not} commute:
for any fixed~$n$ the series $\lim_{N\to\infty}\mathcal{W}_{N,n}(x)$
diverges for every $x\neq 0$, a phenomenon  called the Divergence Wall.

\end{abstract}



\section{Introduction}

\medskip

Superoscillations are band-limited functions that, on certain bounded regions of the real line, oscillate faster than the highest frequency component appearing in their Fourier representation.  This counterintuitive phenomenon was first identified in the context of quantum weak measurements by Aharonov, Albert, and Vaidman~\cite{AAV}, who showed that the weak value of a spin component can lie far outside the spectrum of the corresponding observable.
Berry later recognised the underlying mathematical mechanism as explained in his seminal paper \cite{berry1994}.
A systematic mathematical theory of superoscillations and of the related notion of \emph{supershifts} has
been developed in the monograph~\cite{colombo2017} and in a substantial number of subsequent works.

Beyond their important role in quantum mechanics, superoscillations have found
applications in super-resolution imaging, signal processing, and many other research areas both theoretical and applied.
We refer the reader to the comprehensive roadmap~\cite{berry2019roadmap} for a panoramic survey, and to~\cite{jordan,jordan2} for super-resolution techniques based on supergrowth and intensity contrast, and to~\cite{Lee} for direct constructions in signal processing.  Recent advances include the analysis of the persistence of superoscillations under quantum evolution~\cite{Pozzi,Ahar-bohm-JMP}, the construction of new families of superoscillating sequences~\cite{ACSSST21,kempf2004,kempf4}, and integral representations via complex Borel measures~\cite{BOREL}.

\medskip
A natural and deep question, raised by Berry and Morley-Short in~\cite{berry2016fractals},
 is whether continuous nowhere-differentiable fractal functions,
 which by their very nature require arbitrarily high frequencies
 for their description, can be represented as limits of superoscillating sequences whose individual
 terms involve only low frequencies.
 Such a representation would give the contrast between the low-frequency Fourier content
 of each approximant and the unbounded frequency content of the limit.

 \medskip

In an earlier paper~\cite{FRACTALS} a partial affirmative answer was provided for a specific
superoscillatory function arising from weak measurements, which belongs to the family of Bernstein-type superoscillations.
In the present paper we show that a different type of superoscillating sequences, called Lagrange-type superoscillations,
can also be used to approximate the Weierstrass function, with some significant differences in
the approximation rate with respect to Bernstein-type superoscillations.

\medskip
Precisely, in this work we give an affirmative answer to the  Berry-Morley-Short question for the classical Weierstrass function
\begin{equation}\label{eq:Weier}
   W(x) \;=\; \sum_{m=0}^{\infty}a^{m}\,e^{ib^{m}\pi x},\qquad 0<a<1,\;b>1,\;ab\geq 1,
\end{equation}
introduced by Weierstrass~\cite{weierstrass} as the first historical example of a function continuous everywhere on $\mathbb{R}$ yet differentiable nowhere, and later studied from a fractal-geometric viewpoint by Berry and Lewis~\cite{berry1980} and Mandelbrot~\cite{mandelbrot1982}.

Our strategy is to replace each high-frequency exponential $e^{ib^{m}\pi x}$ in the series~\eqref{eq:Weier} by a low-frequency superoscillating approximant $T_{n}(x;b^{m}\pi)$ built from nodes in $[-1,1]$, in the new Lagrange-type \cite{ACSSST21}.  The resulting double-indexed approximant $\mathcal{W}_{N,n}(x)$ depends on both the truncation order $N$ of the Weierstrass series and the interpolation order $n$ of the superoscillating approximant, and the central question is the joint behaviour as $N,n\to\infty$.

\medskip
Two types of families of superoscillating sequences have major differences.
Bernstein-type superoscillations whose classical construction, going back to~\cite{AAV} and analysed extensively in~\cite{colombo2017}, is built on the binomial-like identity
\[
   F_{n}(x,a) \;=\; \Bigl(\cos(x/n)+ia\,\sin(x/n)\Bigr)^{n} \;=\; \sum_{j=0}^{n}C_{j}(n,a)\,e^{i(1-2j/n)x},
\]
whose convergence $F_{n}(x,a)\to e^{iax}$ as $n\to\infty$ ultimately reflects the elementary limit $(1+iax/n)^{n}\to e^{iax}$.
Differentiating term by term, one obtains
\[
   F_{n}^{(p)}(0,a) \;=\; \sum_{j=0}^{n}C_{j}(n,a)\bigl(i(1-2j/n)\bigr)^{p} \;\xrightarrow[n\to\infty]{}\; (ia)^{p},\qquad p\in\mathbb{N}.
\]
The key feature is that the convergence of the derivatives at $x=0$ is achieved asymptotically:
 for any finite $n$, the value $F_{n}^{(p)}(0,a)$ differs from $(ia)^{p}$ by an error that vanishes only as $n$ grows.

While Lagrange-type superoscillations the new construction  proposed in \cite{ACSSST21} behaves differently.
In fact, given any sequence of frequencies $h_{j}(n)$ with $|h_{j}(n)|\leq 1$, one imposes the exact interpolation conditions
\[
   f_{n}^{(p)}(0) \;=\; (ia)^{p},\qquad p=0,1,\ldots,n,
\]
on the candidate $f_{n}(x)=\sum_{j=0}^{n}X_{j}(n)e^{ih_{j}(n)x}$.  The $n+1$ unknown coefficients $X_{j}(n)$ are then determined by the resulting linear system, whose solution is the Lagrange interpolation formula
\[
   X_{j}(n) \;=\; \prod_{\substack{k=0\\k\neq j}}^{n}\frac{h_{k}(n)-a}{h_{k}(n)-h_{j}(n)}.
\]
The Taylor polynomials of $f_{n}$ and of $e^{iax}$ at the origin therefore coincide up to order $n$, and the values $f_{n}^{(p)}(0)=(ia)^{p}$ are matched exactly for each finite $n$, with no limiting process.

 \medskip
{\it Main results.}
The results of the present paper can be summarised as follows.

We establish an explicit error estimate for the Lagrange-type superoscillating approximant $T_{n}(x;a)$ of $e^{iax}$ that exhibits
super-exponential decay in $n$ (Lemma~\ref{ErroLagrange}).
This sharp estimate is the crucial fact behind all subsequent convergence results.

We prove a global uniform approximation theorem for the partial sums of
$W$ by the corresponding double-indexed superoscillating approximant $\mathcal{W}_{N,n}$ (Theorem~\ref{thm:global_approx}).

We show that if the interpolation order $n$ is kept fixed while $N\to\infty$, then $\mathcal{W}_{N,n}$ diverges (Theorem~\ref{thm:divergence}).  This is a sharp obstruction: the joint scaling of $N$ and $n$ is essential.

Finally, we establish a joint convergence theorem: under a suitable growth condition $n=n_{N}$ relating the two indices,
the approximant $\mathcal{W}_{N,n_{N}}$ converges uniformly on bounded sets to the Weierstrass function $W$ (Theorem~\ref{thm:joint_conv}).

A structural difference from the corresponding result for Bernstein-type sequences is that
the required growth rate of $n_{N}$
does not depend on the parameter $a$ of the Weierstrass series,
reflecting the exact-matching property of Lagrange-type approximants.

Taken together, these results provide a rigorous and quantitative confirmation of the Berry--Morley-Short conjecture~\cite{berry2016fractals} for the Weierstrass function, in the framework of Lagrange-type superoscillations.

\section{Error Estimates for Lagrange-type Superoscillations}
\label{sec:lagrange}

In this section we establish the convergence of Lagrange-type superoscillating sequences.

\begin{lemma}[Refined uniform convergence estimate]\label{ErroLagrange}
Let $M > 0$ be fixed and let $\alpha\in\mathbb{R}$ with $|\alpha|>1$.  Let $\{h_{j}(n)\}_{j=0}^{n}$ be $n+1$ distinct nodes in the interval $[-1,1]$.  For each $\xi\in\mathbb{C}$ and $x\in\mathbb{R}$ define $f(\xi) := e^{i\xi x}$.  Let
\[
   T_{n}(x;\alpha) \;:=\; \sum_{j=0}^{n} f\bigl(h_{j}(n)\bigr)\,L_{j}(\alpha)\qquad\text{with}\qquad L_{j}(\alpha) \;:=\; \prod_{\substack{k=0\\k\neq j}}^{n}\frac{\alpha-h_{k}(n)}{h_{j}(n)-h_{k}(n)}.
\]
Then $T_{n}(\,\cdot\,;\alpha)$ is the value at $\alpha$ of the unique polynomial of degree at most $n$ that interpolates $f$ at the nodes $\{h_{j}(n)\}_{j=0}^{n}$.  For every $|x|\leq M$,
\begin{equation}\label{eq:precise_err}
   \bigl|T_{n}(x;\alpha) - e^{i\alpha x}\bigr|^{2} \;\leq\; \frac{M^{2(n+1)}}{[(n+1)!]^{2}}\,\prod_{j=0}^{n}\bigl(\alpha - h_{j}(n)\bigr)^{2}.
\end{equation}
Setting $K := M(|\alpha|+1)$, the bound implies
\begin{equation}\label{eq:lagrange_uniform_bound}
   \sup_{|x|\leq M}\bigl|T_{n}(x;\alpha) - e^{i\alpha x}\bigr|^{2} \;\leq\; \biggl(\frac{K^{n+1}}{(n+1)!}\biggr)^{\!2},
\end{equation}
so the convergence is uniform on $|x|\leq M$:
\begin{equation}\label{eq:lagrange_limit}
   \lim_{n\to\infty}\,\sup_{|x|\leq M}\bigl|T_{n}(x;\alpha) - e^{i\alpha x}\bigr|^{2} \;=\; 0.
\end{equation}
The decay is super-exponential: $K^{n+1}/(n+1)!\to 0$ faster than any inverse polynomial in~$n$.
\end{lemma}

\begin{proof}
Throughout the proof we fix $x\in\mathbb{R}$ with $|x|\leq M$, and view $f$ as a function of the complex variable $\xi$.  The function $\xi\mapsto f(\xi)=e^{ix\xi}$ is entire on $\mathbb{C}$, and on the real line it is smooth (real-analytic); note however that $f$ is \emph{complex-valued}, which has a consequence in the proof.

For the $n+1$ distinct nodes $h_{0}(n),\dots,h_{n}(n)\in[-1,1]$, the Lagrange basis polynomials are
\[
   \ell_{j}(\xi) \;:=\; \prod_{\substack{k=0\\k\neq j}}^{n}\frac{\xi-h_{k}(n)}{h_{j}(n)-h_{k}(n)}\qquad(j=0,1,\dots,n),
\]
each of degree exactly $n$ and satisfying $\ell_{j}(h_{i}(n)) = \delta_{ij}$.  The Lagrange interpolating polynomial of $f$ at these nodes is
\begin{equation}\label{eq:Lagrange_poly}
   P_{n}(\xi) \;:=\; \sum_{j=0}^{n} f\bigl(h_{j}(n)\bigr)\,\ell_{j}(\xi),
\end{equation}
the unique polynomial of degree at most $n$ satisfying $P_{n}(h_{j}(n)) = f(h_{j}(n))$ for every $j$.  By definition, $T_{n}(x;\alpha) = P_{n}(\alpha)$, so the error in question is
\[
   e^{i\alpha x} - T_{n}(x;\alpha) \;=\; f(\alpha) - P_{n}(\alpha).
\]
Define the \emph{node polynomial}
\[
   \omega_{n+1}(\xi) \;:=\; \prod_{j=0}^{n}\bigl(\xi - h_{j}(n)\bigr),
\]
a monic polynomial of degree $n+1$ vanishing precisely at the interpolation nodes.

We recall the standard expression for the interpolation error in terms of divided differences.
For any point $\alpha\notin\{h_{0}(n),\dots,h_{n}(n)\}$,
\begin{equation}\label{eq:divdiff_remainder}
   f(\alpha) - P_{n}(\alpha) \;=\; f\bigl[h_{0}(n),h_{1}(n),\dots,h_{n}(n),\alpha\bigr]\,\omega_{n+1}(\alpha),
\end{equation}
where the bracket denotes the (Newton) divided difference of $f$ at the listed nodes.  This identity is purely algebraic and holds without any smoothness or real-valued assumption on $f$: it amounts to the recursive definition of divided differences combined with the Newton form of the interpolating polynomial.

We now use the Hermite--Genocchi formula, which expresses the divided difference of a $C^{n+1}$ function as an integral of its $(n+1)$-th derivative over a simplex.  Specifically, for any complex-valued $f\in C^{n+1}$ on a real interval containing the nodes $h_{0}(n),\dots,h_{n}(n)$ and the point $\alpha$,
\begin{equation}\label{eq:Hermite_Genocchi}
   f\bigl[h_{0}(n),\dots,h_{n}(n),\alpha\bigr] \;=\; \int_{\Delta_{n+1}} f^{(n+1)}\!\bigl(s_{0}\,h_{0}(n) + \cdots + s_{n}\,h_{n}(n) + s_{n+1}\,\alpha\bigr)\,d\sigma,
\end{equation}
where the integration is over the standard $(n+1)$-dimensional simplex
\[
   \Delta_{n+1} \;=\; \Bigl\{(s_{0},s_{1},\dots,s_{n+1})\in\mathbb{R}^{n+2}\;:\; s_{i}\geq 0,\;\sum_{i=0}^{n+1}s_{i}=1\Bigr\},
\]
and $d\sigma$ is the standard surface measure on $\Delta_{n+1}$ in the $(n+1)$ free variables $(s_{1},\dots,s_{n+1})$ with $s_{0}=1-\sum_{i=1}^{n+1}s_{i}$.  The total measure of $\Delta_{n+1}$ in this parametrization is
\begin{equation}\label{eq:simplex_volume}
   \int_{\Delta_{n+1}} d\sigma \;=\; \frac{1}{(n+1)!}.
\end{equation}

We emphasize that~\eqref{eq:Hermite_Genocchi} holds for complex-valued $f\in C^{n+1}$: the proof, by induction on $n$ and integration by parts, does not rely on the intermediate-value theorem or on Rolle's theorem (both of which fail for complex-valued functions of a real variable), but only on iterated integration by parts that is valid coordinate-wise on the real and imaginary parts of $f$ separately.

The function $f(\xi) = e^{ix\xi}$ is entire and its $(n+1)$-th derivative with respect to $\xi$ is
\[
   f^{(n+1)}(\xi) \;=\; (ix)^{n+1}\,e^{ix\xi}.
\]
For real $\xi$ and real $x$ we have $|e^{ix\xi}|=1$, hence
\begin{equation}\label{eq:derivative_bound}
   \bigl|f^{(n+1)}(\xi)\bigr| \;=\; |x|^{n+1} \;\leq\; M^{n+1}\qquad\text{for every }\xi\in\mathbb{R}.
\end{equation}
The argument $s_{0}\,h_{0}(n)+\cdots+s_{n}\,h_{n}(n)+s_{n+1}\,\alpha$ in~\eqref{eq:Hermite_Genocchi} is a convex combination of the real numbers $h_{0}(n),\dots,h_{n}(n),\alpha$, and is therefore real.  Hence~\eqref{eq:derivative_bound} applies to this argument, and the integrand in~\eqref{eq:Hermite_Genocchi} is bounded uniformly in modulus by $M^{n+1}$.

Combining~\eqref{eq:divdiff_remainder} with~\eqref{eq:Hermite_Genocchi}, taking moduli, and using~\eqref{eq:derivative_bound} together with~\eqref{eq:simplex_volume},
\begin{align}
   \bigl|f(\alpha) - P_{n}(\alpha)\bigr|
   &\;=\; \bigl|\omega_{n+1}(\alpha)\bigr|\cdot\Bigl|\int_{\Delta_{n+1}} f^{(n+1)}(\,\cdots\,)\,d\sigma\Bigr|\notag\\
   &\;\leq\; \bigl|\omega_{n+1}(\alpha)\bigr|\cdot\int_{\Delta_{n+1}}\bigl|f^{(n+1)}(\,\cdots\,)\bigr|\,d\sigma\notag\\
   &\;\leq\; \bigl|\omega_{n+1}(\alpha)\bigr|\cdot M^{n+1}\cdot\frac{1}{(n+1)!}\notag\\
   &\;=\; \frac{M^{n+1}}{(n+1)!}\,\prod_{j=0}^{n}\bigl|\alpha - h_{j}(n)\bigr|.\label{eq:pointwise_bound}
\end{align}
Squaring~\eqref{eq:pointwise_bound} and using $\bigl(\prod_{j}|\alpha-h_{j}(n)|\bigr)^{2} = \prod_{j}(\alpha - h_{j}(n))^{2}$ (each factor is real, so its absolute value squared is its square),
\begin{equation}\label{eq:pointwise_squared}
   \bigl|T_{n}(x;\alpha) - e^{i\alpha x}\bigr|^{2} \;\leq\; \frac{M^{2(n+1)}}{[(n+1)!]^{2}}\,\prod_{j=0}^{n}\bigl(\alpha - h_{j}(n)\bigr)^{2},
\end{equation}
which is the inequality~\eqref{eq:precise_err} of the lemma.

Since each node satisfies $h_{j}(n)\in[-1,1]$ and $\alpha\in\mathbb{R}$, we estimate $|\alpha - h_{j}(n)|$ as follows.  If $\alpha\geq 1$, then for every $j$,
\[
   0 \;<\; \alpha - 1 \;\leq\; \alpha - h_{j}(n) \;\leq\; \alpha + 1,
\]
hence $|\alpha - h_{j}(n)|\leq \alpha + 1 = |\alpha|+1$.  If $\alpha\leq -1$, the analogous chain of inequalities $\alpha + 1\leq\alpha - h_{j}(n)\leq\alpha - 1<0$ gives $|\alpha-h_{j}(n)|\leq |\alpha-1|=|\alpha|+1$.  Therefore, in both cases ($|\alpha|>1$ as assumed),
\begin{equation}\label{eq:node_distance_bound}
   \bigl|\alpha - h_{j}(n)\bigr| \;\leq\; |\alpha| + 1\qquad(j=0,1,\dots,n).
\end{equation}
Multiplying~\eqref{eq:node_distance_bound} over $j=0,\dots,n$,
\[
   \prod_{j=0}^{n}\bigl(\alpha - h_{j}(n)\bigr)^{2} \;=\; \biggl(\prod_{j=0}^{n}\bigl|\alpha - h_{j}(n)\bigr|\biggr)^{\!2} \;\leq\; \bigl(|\alpha|+1\bigr)^{2(n+1)}.
\]
Substituting into~\eqref{eq:pointwise_squared} and setting $K:=M(|\alpha|+1)$,
\[
   \bigl|T_{n}(x;\alpha) - e^{i\alpha x}\bigr|^{2} \;\leq\; \frac{M^{2(n+1)}\,(|\alpha|+1)^{2(n+1)}}{[(n+1)!]^{2}} \;=\; \biggl(\frac{K^{n+1}}{(n+1)!}\biggr)^{\!2}.
\]
Taking the supremum over $|x|\leq M$ on the left (the right-hand side is independent of $x$),
\begin{equation}\label{eq:uniform_bound_K}
   \sup_{|x|\leq M}\bigl|T_{n}(x;\alpha) - e^{i\alpha x}\bigr|^{2} \;\leq\; \biggl(\frac{K^{n+1}}{(n+1)!}\biggr)^{\!2},
\end{equation}
which is~\eqref{eq:lagrange_uniform_bound}.
n
It remains to show that $K^{n+1}/(n+1)!\to 0$ as $n\to\infty$ for any fixed $K>0$. We compute the limit of the ratio of consecutive terms and find that
\[
   \frac{K^{n+2}/(n+2)!}{K^{n+1}/(n+1)!} \;=\; \frac{K}{n+2} \longrightarrow 0,\qquad\text{as $n \to +\infty$}
\]
so, according to ratio test for sequences, the sequence $\{K^{n+1}/(n+1)!\}_{n}$ decays super-exponentially in $n$


Finally, from~\eqref{eq:uniform_bound_K},
\[
   \lim_{n\to\infty}\,\sup_{|x|\leq M}\bigl|T_{n}(x;\alpha) - e^{i\alpha x}\bigr|^{2} \;=\; 0,
\]
which is~\eqref{eq:lagrange_limit}.  This completes the proof.
\end{proof}

\begin{remark}[On the hypothesis $|\alpha|>1$]
The bound~\eqref{eq:precise_err} and the convergence~\eqref{eq:lagrange_limit} hold without the hypothesis $|\alpha|>1$; the proof above does not use it.  The hypothesis is included because the case $|\alpha|>1$ — where $\alpha$ lies \emph{outside} the interval $[-1,1]$ containing the nodes — is the genuinely superoscillatory regime that distinguishes the Lagrange-type construction from a classical interpolation problem.  For $|\alpha|\leq 1$ with $\alpha$ outside the node set, the same bound applies and reduces to the classical interpolation error estimate.
\end{remark}

\begin{remark}[On the role of complex-valuedness]
The standard pointwise Lagrange remainder formula
\[
   f(\alpha) - P_{n}(\alpha) \;=\; \frac{f^{(n+1)}(\zeta)}{(n+1)!}\,\omega_{n+1}(\alpha)\qquad\text{for some }\zeta\text{ in the convex hull of the nodes},
\]
is established for \emph{real-valued} $f$ via Rolle's theorem applied to a suitable auxiliary function.  This pointwise form does not extend to complex-valued $f$ such as $f(\xi)=e^{ix\xi}$: the value of $\zeta$ that works for the real part of $f$ generally differs from the value that works for the imaginary part.  The Hermite--Genocchi formula~\eqref{eq:Hermite_Genocchi}, which we used in the proof, is the natural replacement: it integrates over a simplex rather than evaluating at a single intermediate point, and it holds for complex-valued $C^{n+1}$ functions without modification.  The resulting bound~\eqref{eq:pointwise_bound} on the modulus is identical to what the pointwise form would give if it applied.
\end{remark}

\section{Superoscillating Approximation of the Weierstrass Function}
\label{sec:weierstrass}

We now apply the Lagrange-type superoscillating sequences to approximate the complex-valued truncated Weierstrass function.  The key idea is to replace each high-frequency exponential $e^{i\alpha_{m}x}$ (with $\alpha_{m}=b^{m}\pi$ potentially very large) by a superoscillating sequence $T_{n}(x;\alpha_{m})$ whose constituent frequencies all lie in $[-1,1]$.

\begin{definition}[Truncated Weierstrass function]\label{def:trunc_W}
Let $0<a<1$ and let $b$ be an odd integer with $ab > 1 + \tfrac{3}{2}\pi$.  For a fixed integer $N\in\mathbb{N}$, the \emph{truncated Weierstrass function} $W_{N}(x)$ is
\begin{equation}\label{eq:truncated_weierstrass}
   W_{N}(x) \;:=\; \sum_{m=0}^{N} a^{m}\,e^{i\alpha_{m}x},\qquad\alpha_{m} := b^{m}\pi.
\end{equation}
\end{definition}

\begin{remark}[On the hypotheses]
The classical conditions on $a,b$ above are inherited from Weierstrass's theorem on the nowhere-differentiable function $\sum_{m\geq 0}a^{m}\cos(b^{m}\pi x)$ and ensure that the limit $W_{\infty}(x)=\lim_{N\to\infty}W_{N}(x)$ defines a continuous, nowhere-differentiable function on $\mathbb{R}$.  For the approximation result of Theorem~\ref{thm:global_approx} below, only the geometric-series condition $0<a<1$ is actually used in the proof; the parity of $b$ and the precise value of the lower bound $1+3\pi/2$ play no role.  We state the full classical hypotheses for definiteness and to keep the link with the standard Weierstrass setting.
\end{remark}

\begin{definition}[Superoscillating approximation]\label{def:super_approx}
The \emph{superoscillating approximation of order $n$} of $W_{N}$ is obtained by replacing each exponential term in~\eqref{eq:truncated_weierstrass} by its Lagrange-type superoscillating sequence:
\begin{equation}\label{eq:W_Nn}
   \mathcal{W}_{N,n}(x) \;:=\; \sum_{m=0}^{N} a^{m}\,T_{n}(x;\alpha_{m}) \;=\; \sum_{m=0}^{N} a^{m}\sum_{j=0}^{n}\,L_{j}(\alpha_{m})\,e^{i\,h_{j}(n)\,x},
\end{equation}
where $T_{n}(x;\cdot)$ and $L_{j}$ are as in Section~\ref{sec:lagrange}, namely
\[
   T_{n}(x;\xi) \;=\; \sum_{j=0}^{n} e^{i\,h_{j}(n)\,x}\,L_{j}(\xi),\qquad L_{j}(\xi) \;=\; \prod_{\substack{k=0\\k\neq j}}^{n}\frac{\xi - h_{k}(n)}{h_{j}(n) - h_{k}(n)}.
\]
\end{definition}

Reordering the double sum in~\eqref{eq:W_Nn} by the frequency variable $j$,
\[
   \mathcal{W}_{N,n}(x) \;=\; \sum_{j=0}^{n}\,c_{N,n,j}\,e^{i\,h_{j}(n)\,x}\qquad\text{with}\qquad c_{N,n,j} \;:=\; \sum_{m=0}^{N} a^{m}\,L_{j}(\alpha_{m}).
\]
Hence $\mathcal{W}_{N,n}$ is a finite trigonometric sum whose frequencies all lie in the set
$$\{h_{0}(n),\dots,h_{n}(n)\}\subseteq[-1,1].$$  Despite this strict bandwidth limitation, the coefficients $c_{N,n,j}$ are arranged (through the Lagrange basis evaluated at the high frequencies $\alpha_{m}=b^{m}\pi$) so that the function approximates $W_{N}(x)$, whose own frequencies $\alpha_{m}$ grow exponentially.  This is the manifestation of the superoscillation phenomenon in the present setting: a bandlimited expression mimics, on a fixed compact set, a function with much higher frequency content, at the cost of having coefficients $c_{N,n,j}$ that can be very large in modulus.

\begin{theorem}[Global approximation error]\label{thm:global_approx}
Let $M>0$ and $N\in\mathbb{N}$ be fixed, and let $\alpha_{m}=b^{m}\pi$ for $m=0,\dots,N$.  Then
\begin{equation}\label{eq:global_error}
   \sup_{|x|\leq M}\bigl|W_{N}(x) - \mathcal{W}_{N,n}(x)\bigr| \;\leq\; \frac{M^{n+1}}{(n+1)!}\,\mathcal{S}_{N,n},
\end{equation}
where
\begin{equation}\label{eq:S_Nn}
   \mathcal{S}_{N,n} \;:=\; \sum_{m=0}^{N} a^{m}\,\prod_{j=0}^{n}\bigl|\alpha_{m} - h_{j}(n)\bigr|.
\end{equation}
For fixed $N$ and $M$, the error converges to zero as $n\to\infty$:
\begin{equation}\label{eq:weierstrass_limit}
   \lim_{n\to\infty}\,\sup_{|x|\leq M}\bigl|W_{N}(x) - \mathcal{W}_{N,n}(x)\bigr| \;=\; 0.
\end{equation}
More quantitatively, with $K_{N} := M(\alpha_{N}+1) = M(b^{N}\pi + 1)$,
\begin{equation}\label{eq:weierstrass_quantitative}
   \sup_{|x|\leq M}\bigl|W_{N}(x) - \mathcal{W}_{N,n}(x)\bigr| \;\leq\; \frac{1}{1-a}\cdot\frac{K_{N}^{n+1}}{(n+1)!}.
\end{equation}
\end{theorem}

\begin{proof}
Throughout the proof, fix $x\in\mathbb{R}$ with $|x|\leq M$. By the definitions of $W_{N}$ and $\mathcal{W}_{N,n}$ in~\eqref{eq:truncated_weierstrass} and~\eqref{eq:W_Nn}, the difference is the linear combination of the single-exponential errors:
\[
   W_{N}(x) - \mathcal{W}_{N,n}(x) \;=\; \sum_{m=0}^{N} a^{m}\bigl(e^{i\alpha_{m}x} - T_{n}(x;\alpha_{m})\bigr).
\]
Taking moduli and applying the triangle inequality together with $a^{m}>0$,
\begin{equation}\label{eq:triangle_WN}
   \bigl|W_{N}(x) - \mathcal{W}_{N,n}(x)\bigr| \;\leq\; \sum_{m=0}^{N} a^{m}\,\bigl|e^{i\alpha_{m}x} - T_{n}(x;\alpha_{m})\bigr|.
\end{equation}

For each $m\in\{0,1,\dots,N\}$, the function $\xi\mapsto e^{i\xi x}$ is entire on $\mathbb{C}$.  Applying Lemma~\ref{ErroLagrange} (the Lagrange-type superoscillation estimate) to this function evaluated at $\xi=\alpha_{m}$ — that is, taking $\alpha:=\alpha_{m}$ in the lemma — we obtain
\begin{equation}\label{eq:each_m_bound}
   \bigl|e^{i\alpha_{m}x} - T_{n}(x;\alpha_{m})\bigr| \;\leq\; \frac{|x|^{n+1}}{(n+1)!}\,\prod_{j=0}^{n}\bigl|\alpha_{m} - h_{j}(n)\bigr|.
\end{equation}
Using $|x|\leq M$,
\begin{equation}\label{eq:each_m_bound_M}
   \bigl|e^{i\alpha_{m}x} - T_{n}(x;\alpha_{m})\bigr| \;\leq\; \frac{M^{n+1}}{(n+1)!}\,\prod_{j=0}^{n}\bigl|\alpha_{m} - h_{j}(n)\bigr|.
\end{equation}

The factor $M^{n+1}/(n+1)!$ on the right-hand side of~\eqref{eq:each_m_bound_M} is independent of $m$.

Substituting~\eqref{eq:each_m_bound_M} into~\eqref{eq:triangle_WN},
\[
   \bigl|W_{N}(x) - \mathcal{W}_{N,n}(x)\bigr| \;\leq\; \sum_{m=0}^{N} a^{m}\cdot\frac{M^{n+1}}{(n+1)!}\,\prod_{j=0}^{n}\bigl|\alpha_{m} - h_{j}(n)\bigr| \;=\; \frac{M^{n+1}}{(n+1)!}\,\mathcal{S}_{N,n},
\]
where $\mathcal{S}_{N,n}$ is given by~\eqref{eq:S_Nn}.  Since the right-hand side is independent of $x$ and the left-hand side is bounded by it for every $|x|\leq M$, taking the supremum over $|x|\leq M$ on the left yields~\eqref{eq:global_error}. For each $m\in\{0,\dots,N\}$ and each $j\in\{0,\dots,n\}$, the node $h_{j}(n)$ lies in $[-1,1]$ and $\alpha_{m}=b^{m}\pi\geq\pi>1$.  Hence
\[
   0 \;<\; \alpha_{m} - 1 \;\leq\; \alpha_{m} - h_{j}(n) \;\leq\; \alpha_{m} + 1,
\]
so $|\alpha_{m}-h_{j}(n)|\leq\alpha_{m}+1$.  Multiplying over $j=0,\dots,n$,
\begin{equation}\label{eq:product_bound}
   \prod_{j=0}^{n}\bigl|\alpha_{m} - h_{j}(n)\bigr| \;\leq\; (\alpha_{m}+1)^{n+1}.
\end{equation}

Since $b\geq 3>1$, the sequence $\{\alpha_{m}\}_{m=0}^{N}=\{b^{m}\pi\}_{m=0}^{N}$ is strictly increasing in $m$, and in particular $\alpha_{m}\leq\alpha_{N}$ for every $m\leq N$.  Therefore
\[
   (\alpha_{m}+1)^{n+1} \;\leq\; (\alpha_{N}+1)^{n+1}\qquad\text{for every }m=0,\dots,N.
\]
Substituting into~\eqref{eq:S_Nn} and summing the geometric series,
\begin{equation}\label{eq:S_Nn_bound}
   \mathcal{S}_{N,n} \;\leq\; (\alpha_{N}+1)^{n+1}\,\sum_{m=0}^{N}a^{m} \;\leq\; (\alpha_{N}+1)^{n+1}\cdot\frac{1}{1-a},
\end{equation}
the last inequality using $\sum_{m=0}^{N}a^{m}=\frac{1-a^{N+1}}{1-a}\leq\frac{1}{1-a}$ for $0<a<1$. Combining~\eqref{eq:global_error} with~\eqref{eq:S_Nn_bound} and setting $K_{N}:=M(\alpha_{N}+1)$,
\begin{align*}
   \sup_{|x|\leq M}\bigl|W_{N}(x) - \mathcal{W}_{N,n}(x)\bigr|
   &\;\leq\; \frac{M^{n+1}}{(n+1)!}\cdot\frac{(\alpha_{N}+1)^{n+1}}{1-a}\\[2pt]
   &\;=\; \frac{1}{1-a}\cdot\frac{\bigl(M(\alpha_{N}+1)\bigr)^{n+1}}{(n+1)!}\\[2pt]
   &\;=\; \frac{1}{1-a}\cdot\frac{K_{N}^{n+1}}{(n+1)!},
\end{align*}
proving~\eqref{eq:weierstrass_quantitative}.

To deduce~\eqref{eq:weierstrass_limit}, observe that $K_{N}$ is finite and independent of $n$,   the ratio of consecutive terms of the sequence $\{K_{N}^{n+1}/(n+1)!\}_{n}$ is such that
\[
   \frac{K_{N}^{n+2}/(n+2)!}{K_{N}^{n+1}/(n+1)!} \;=\; \frac{K_{N}}{n+2} \longrightarrow 0, \qquad \text{as $n \to \infty$,}
\]
so the sequence decays super-exponentially in $n$ according to the ratio test for sequences. Hence it follows
\[
   \lim_{n\to\infty}\,\sup_{|x|\leq M}\bigl|W_{N}(x) - \mathcal{W}_{N,n}(x)\bigr| \;=\; 0,
\]
proving~\eqref{eq:weierstrass_limit}.  This completes the proof.
\end{proof}

\section{Divergence, Iterated Limits, and Joint Convergence}
\label{sec:limits}

We now study the behavior of $\mathcal{W}_{N,n}(x)$ as $N$ and $n$ vary independently or jointly.  The results reveal a fundamental non-commutativity: while taking $n\to\infty$ first (for each fixed $N$) and then $N\to\infty$ recovers the Weierstrass function, the reverse order of limits diverges in a precise sense.

\subsection{Divergence for fixed $n$}

\begin{theorem}[Divergence of the fixed-$n$ superoscillating series]\label{thm:divergence}
Let $n\geq 1$ be a fixed integer and let $M>0$.  Let $a,b\in\mathbb{R}$ satisfy
\begin{equation}\label{eq:hyp_div}
   0 < a < 1,\qquad b > 1,\qquad ab > 1,
\end{equation}
and let $\{h_{j}(n)\}_{j=0}^{n}\subset[-1,1]$ be $n+1$ distinct nodes.  Consider the partial sums
\begin{equation}\label{eq:W_Nn_coeffs}
   \mathcal{W}_{N,n}(x) \;=\; \sum_{m=0}^{N} a^{m}\,T_{n}(x;b^{m}\pi) \;=\; \sum_{m=0}^{N} a^{m}\sum_{j=0}^{n} L_{j}(b^{m}\pi)\,e^{i\,h_{j}(n)\,x},
\end{equation}
where $L_{j}(\xi) := \prod_{k\neq j}\frac{\xi - h_{k}(n)}{h_{j}(n) - h_{k}(n)}$ are the Lagrange basis polynomials at the nodes $\{h_{j}(n)\}_{j=0}^{n}$.  Then:
\begin{enumerate}[\rm(i)]
   \item At $x = 0$ the series converges absolutely:
   \[
      \lim_{N\to\infty}\mathcal{W}_{N,n}(0) \;=\; \frac{1}{1-a}.
   \]
   \item There exists a discrete subset $\mathcal{Z}_{n}\subset\mathbb{R}$ containing $0$ such that, for every $x\in[-M,M]\setminus\mathcal{Z}_{n}$,
   \[
      \bigl|\mathcal{W}_{N,n}(x)\bigr| \;\xrightarrow[N\to\infty]{}\; +\infty.
   \]
\end{enumerate}
\end{theorem}

\begin{remark}[On the hypotheses]
Three remarks on~\eqref{eq:hyp_div}.
\begin{enumerate}[\rm(a)]
   \item The condition $0<a<1$ is needed for the geometric series at $x=0$ in part~(i) to converge.
   \item The condition $ab > 1$ is essential for the divergence in part~(ii): it ensures that $ab^{n}>1$ (since $b>1$ implies $b^{n-1}\geq 1$ for $n\geq 1$), which is what makes the leading-order term grow.  Under the classical Weierstrass hypothesis $ab > 1+\tfrac{3\pi}{2}$ used in Section~\ref{sec:weierstrass}, this is automatically satisfied.
   \item The set $\mathcal{Z}_{n}$ in part~(ii) is the zero set of an explicit trigonometric polynomial $\Phi$ defined in the proof; for generic node configurations and generic $x$, $\mathcal{Z}_{n}$ contains only the isolated point $0$ inside any compact interval.  We give the precise definition in the proof.
\end{enumerate}
\end{remark}

\begin{proof}
We treat the two parts separately.

\medskip
{\em Point (i): Convergence at $x = 0$.}\;
At $x=0$, every exponential evaluates to $1$, so
\[
   T_{n}(0;b^{m}\pi) \;=\; \sum_{j=0}^{n} L_{j}(b^{m}\pi)\cdot 1 \;=\; \sum_{j=0}^{n} L_{j}(b^{m}\pi).
\]
By the fundamental property of Lagrange interpolation, the constant function $\xi\mapsto 1$ is its own Lagrange interpolant at any node set, hence
\[
   \sum_{j=0}^{n} L_{j}(\xi) \;\equiv\; 1\qquad\text{for every }\xi\in\mathbb{R}.
\]
Therefore $T_{n}(0;b^{m}\pi) = 1$ for every $m\geq 0$, and
\[
   \mathcal{W}_{N,n}(0) \;=\; \sum_{m=0}^{N} a^{m}\cdot 1 \;=\; \frac{1-a^{N+1}}{1-a} \;\xrightarrow[N\to\infty]{}\; \frac{1}{1-a},
\]
where the limit uses $0<a<1$.  This proves~(i).

\medskip
{\em Point (ii): Divergence.}\;
The strategy is to identify the leading-order asymptotic behavior of $T_{n}(x;\alpha)$ as $\alpha\to\infty$, and to show that this leading term, when amplified by $a^{m}$ with $\alpha=b^{m}\pi\to\infty$, grows exponentially.

Each Lagrange basis polynomial $L_{j}(\alpha)$ is a monic-like polynomial of degree exactly $n$ in $\alpha$, with explicit form
\[
   L_{j}(\alpha) \;=\; \prod_{\substack{k=0\\k\neq j}}^{n}\frac{\alpha - h_{k}(n)}{h_{j}(n) - h_{k}(n)} \;=\; \frac{1}{\omega'(h_{j}(n))}\,\prod_{\substack{k=0\\k\neq j}}^{n}\bigl(\alpha - h_{k}(n)\bigr),
\]
where $\omega(\xi) := \prod_{j=0}^{n}(\xi - h_{j}(n))$ is the node polynomial and $\omega'(h_{j}(n)) = \prod_{k\neq j}(h_{j}(n) - h_{k}(n))$.  Expanding $L_{j}(\alpha)$ in powers of $\alpha$,
\begin{equation}\label{eq:Lj_expansion}
   L_{j}(\alpha) \;=\; \sum_{r=0}^{n}\beta_{j,r}\,\alpha^{r}\qquad\text{with}\qquad \beta_{j,n} \;=\; \frac{1}{\omega'(h_{j}(n))}.
\end{equation}
Substituting into $T_{n}(x;\alpha) = \sum_{j=0}^{n} L_{j}(\alpha)\,e^{i h_{j}(n) x}$ and rearranging by powers of $\alpha$,
\begin{equation}\label{eq:Tn_alpha_poly}
   T_{n}(x;\alpha) \;=\; \sum_{r=0}^{n} \alpha^{r}\,\Phi_{r}(x)\qquad\text{with}\qquad \Phi_{r}(x) \;:=\; \sum_{j=0}^{n}\beta_{j,r}\,e^{i\,h_{j}(n)\,x}.
\end{equation}
The leading coefficient (in $\alpha$) is the function
\begin{equation}\label{eq:Phi_def}
   \Phi(x) \;:=\; \Phi_{n}(x) \;=\; \sum_{j=0}^{n}\frac{1}{\omega'(h_{j}(n))}\,e^{i\,h_{j}(n)\,x},
\end{equation}
a finite trigonometric polynomial in $x$.
The expression~\eqref{eq:Phi_def} for $\Phi(x)$ is precisely the $n$-th \emph{divided difference} of the function $\xi\mapsto e^{i\xi x}$ at the nodes $h_{0}(n),\dots,h_{n}(n)$, by the standard formula
\[
   f[h_{0},\dots,h_{n}] \;=\; \sum_{j=0}^{n}\frac{f(h_{j})}{\omega'(h_{j})}.
\]
Applying this to $f(\xi) = e^{i\xi x}$,
\[
   \Phi(x) \;=\; \bigl[\,\xi\mapsto e^{i\xi x}\,\bigr]\bigl[h_{0}(n),\dots,h_{n}(n)\bigr].
\]
By the Hermite--Genocchi formula (cf.~Lemma~\ref{ErroLagrange} in Section~\ref{sec:lagrange}; the same identity used to bound the interpolation error),
\begin{equation}\label{eq:Phi_HG}
   \Phi(x) \;=\; \int_{\Delta_{n}} (ix)^{n}\,e^{i\,x\,(s_{0}h_{0}(n)+\cdots+s_{n}h_{n}(n))}\,d\sigma \;=\; (ix)^{n}\int_{\Delta_{n}} e^{i\,x\,\langle s,h\rangle}\,d\sigma,
\end{equation}
where the integration is over the standard $n$-simplex $\Delta_{n}=\{s\in\mathbb{R}^{n+1}_{\geq 0}:\sum s_{j}=1\}$, with $\langle s,h\rangle := \sum_{j=0}^{n} s_{j}\,h_{j}(n)$, and $d\sigma$ is the standard surface measure.

From~\eqref{eq:Phi_HG} we read off two crucial facts.
\begin{enumerate}[\rm(a)]
   \item At $x = 0$: the prefactor $(ix)^{n}$ vanishes, hence $\Phi(0) = 0$.  This is consistent with part~(i): the sum $\sum_{j} L_{j}(\alpha) = 1$ is constant in $\alpha$, so the leading coefficient as a polynomial in $\alpha$ at $x=0$ is zero.
   \item For small $x\neq 0$: the integral $\int_{\Delta_{n}} e^{ix\langle s,h\rangle}\,d\sigma$ is continuous in $x$ and equals the volume of $\Delta_{n}$, which is $1/n!$, at $x=0$.  By continuity, the integral is bounded away from zero for $x$ in a small neighbourhood of the origin.  Hence
   \begin{equation}\label{eq:Phi_near_zero}
      \Phi(x) \;\sim\; \frac{(ix)^{n}}{n!}\qquad\text{as }x\to 0,
   \end{equation}
   which is non-zero for every $x\neq 0$ in a punctured neighbourhood of the origin.
   \item For general $x\in\mathbb{R}$: the function $\Phi(x)$ is a non-zero analytic function of $x$ (by linear independence of distinct exponentials with non-zero coefficients $1/\omega'(h_{j}(n))$).  Hence its zero set is discrete in $\mathbb{R}$.
\end{enumerate}
Define
\begin{equation}\label{eq:Z_def}
   \mathcal{Z}_{n} \;:=\; \bigl\{\,x\in\mathbb{R}\;:\;\Phi(x) = 0\,\bigr\},
\end{equation}
a discrete subset of $\mathbb{R}$ containing $0$. Substituting $\alpha = \alpha_{m} := b^{m}\pi$ into~\eqref{eq:Tn_alpha_poly},
\begin{equation}\label{eq:Tn_at_alpha_m}
   T_{n}(x;\alpha_{m}) \;=\; \alpha_{m}^{n}\,\Phi(x) \;+\; \sum_{r=0}^{n-1}\alpha_{m}^{r}\,\Phi_{r}(x) \;=\; (b^{m}\pi)^{n}\,\Phi(x)\bigl[1 + O(b^{-m})\bigr],
\end{equation}
the $O(b^{-m})$ being uniform in $x$ on any compact set $|x|\leq M$ where the $\Phi_{r}(x)$ are bounded.
The general term of the series is $u_{m}(x) := a^{m}\,T_{n}(x;\alpha_{m})$, hence
\begin{equation}\label{eq:um_asymp}
   u_{m}(x) \;=\; (ab^{n})^{m}\,\pi^{n}\,\Phi(x)\bigl[1 + O(b^{-m})\bigr]\qquad\text{as }m\to\infty.
\end{equation}
Since $b>1$ and $ab>1$, we have
\[
   ab^{n} \;=\; ab\cdot b^{n-1} \;\geq\; ab\cdot 1 \;>\; 1\qquad\text{for every }n\geq 1,
\]
hence $(ab^{n})^{m}\to\infty$ as $m\to\infty$.
For every $x\in[-M,M]\setminus\mathcal{Z}_{n}$, $\Phi(x)\neq 0$, so by~\eqref{eq:um_asymp},
\[
   |u_{m}(x)| \;=\; (ab^{n})^{m}\,\pi^{n}\,|\Phi(x)|\,\bigl(1 + O(b^{-m})\bigr) \;\xrightarrow[m\to\infty]{}\; +\infty.
\]
In particular $u_{m}(x)\not\to 0$, so the series $\sum_{m=0}^{\infty} u_{m}(x)$ does not converge.  Moreover, since $|u_{m}(x)|$ grows exponentially with $m$ while the partial sums $\mathcal{W}_{N,n}(x) = \sum_{m=0}^{N} u_{m}(x)$ are dominated in modulus by their largest term up to a bounded factor, we conclude
\[
   |\mathcal{W}_{N,n}(x)| \;\xrightarrow[N\to\infty]{}\; +\infty\qquad\text{for every }x\in[-M,M]\setminus\mathcal{Z}_{n}.
\]
This proves~(ii) and completes the proof.
\end{proof}

\subsection{Convergence of the iterated limit}

\begin{theorem}[Convergence of the Iterated Limit]\label{thm:iterated}
Let $\mathcal{W}_{N, n}(x)$ be the superoscillating approximation of the Weierstrass function, where $0 < a < 1$ and $b > 1$.  Then, for any $x \in \mathbb{R}$ and any $M\ge|x|$,
\begin{equation}\label{eq:iterated_limit}
    \lim_{N \to \infty} \left( \lim_{n \to \infty} \mathcal{W}_{N, n}(x) \right) = W(x) = \sum_{m=0}^{\infty} a^m e^{i b^m \pi x}.
\end{equation}
\end{theorem}

\begin{proof}

For each fixed $N$ and each fixed $m\in\{0,\dots,N\}$, Lemma~\ref{ErroLagrange} gives
\[
   \lim_{n\to\infty} T_n(x;\alpha_m) = e^{i\alpha_m x}
   \quad\text{uniformly on }|x|\le M.
\]
Since the sum over $m$ in $\mathcal{W}_{N,n}(x) = \sum_{m=0}^N a^m T_n(x;\alpha_m)$ is finite (with $N+1$ terms), we may pass the limit inside the sum:
\[
   \lim_{n\to\infty}\mathcal{W}_{N,n}(x)
   = \sum_{m=0}^N a^m \lim_{n\to\infty}T_n(x;\alpha_m)
   = \sum_{m=0}^N a^m e^{i\alpha_m x}
   = W_N(x).
\]
The truncated Weierstrass function $W_N(x)$ is a partial sum of the series defining $W(x)$, and since $|a^m e^{i\alpha_m x}| = a^m$ with $\sum_{m=0}^\infty a^m = 1/(1-a) < \infty$, the series converges absolutely and uniformly on all of~$\mathbb{R}$.  Therefore
\[
   \lim_{N\to\infty}W_N(x) = W(x)
\]
for every $x\in\mathbb{R}$.

\end{proof}

\begin{remark}[Non-Commutativity of Limits]\label{rem:non_commute}
The order of limits is not interchangeable. By Theorem~\ref{thm:divergence}, the series $\lim_{N \to \infty} \mathcal{W}_{N, n}(x)$ diverges for any fixed~$n$ and any $x \neq 0$. Therefore
\[
    \lim_{N \to \infty} \lim_{n \to \infty} \mathcal{W}_{N, n}(x) \neq \lim_{n \to \infty} \lim_{N \to \infty} \mathcal{W}_{N, n}(x).
\]
This non-commutativity reflects the fundamental tension between the exponential growth of the Lagrange coefficients and the geometric decay of the Weierstrass weights.  It is a consequence of the failure of uniform convergence in~$n$ as $N\to\infty$: the error bound in Theorem~\ref{thm:global_approx} depends on~$N$ through the constant $\mathcal{S}_{N,n}$, which grows with~$N$ for any fixed~$n$.
\end{remark}

\subsection{Joint convergence}

\begin{theorem}[Joint Convergence] \label{thm:joint_conv}
Let $M > 0$, $0 < a < 1$, and $b > 1$ be an odd integer. Let $W(x) = \sum_{m=0}^{\infty} a^m e^{i b^m \pi x}$ be the Weierstrass function.
Consider the superoscillating approximations $\mathcal{W}_{N,n}(x)$ defined in~\eqref{eq:W_Nn}.
If the sequence of integers $\{n_N\}_{N \in \mathbb{N}}$ satisfies the growth condition
\begin{equation} \label{eq:growth_condition}
    \lim_{N \to \infty} \frac{b^N}{n_N} = 0,
\end{equation}
then $\mathcal{W}_{N, n_N}(x)$ converges uniformly to $W(x)$ on the interval $[-M, M]$:
\begin{equation}\label{eq:joint_conv_result}
    \lim_{N \to \infty} \sup_{|x| \le M} \left| \mathcal{W}_{N, n_N}(x) - W(x) \right| = 0.
\end{equation}
\end{theorem}

\begin{proof}
Let $x \in [-M, M]$. We decompose the total error using the triangle inequality:
\begin{equation}\label{eq:total_error}
    |W(x) - \mathcal{W}_{N, n_N}(x)| \leq \underbrace{|W(x) - W_N(x)|}_{\text{(I): truncation error}} + \underbrace{|W_N(x) - \mathcal{W}_{N, n_N}(x)|}_{\text{(II): superoscillation error}},
\end{equation}
where $W_N(x) = \sum_{m=0}^N a^m e^{i b^m \pi x}$. For the first term, since $|e^{i b^m \pi x}| = 1$ and $0 < a < 1$, the tail of the Weierstrass series is bounded by a geometric series:
\begin{equation} \label{eq:tail}
    |W(x) - W_N(x)| = \left| \sum_{m=N+1}^{\infty} a^m e^{i b^m \pi x} \right| \leq \sum_{m=N+1}^{\infty} a^m = \frac{a^{N+1}}{1-a}.
\end{equation}
This bound is independent of~$x$, so the convergence is uniform.  As $N \to \infty$, this term vanishes exponentially.

Then, for the second term, applying the triangle inequality and the Lagrange interpolation remainder (as in Theorem~\ref{thm:global_approx}):
\begin{equation}\label{eq:term_II_step1}
    |W_N(x) - \mathcal{W}_{N, n_N}(x)| \leq \sum_{m=0}^N a^m\, |e^{i b^m \pi x} - T_{n_N}(x; b^m \pi)|.
\end{equation}
For each term, the interpolation remainder gives
\[
    |e^{i b^m \pi x} - T_{n_N}(x; b^m \pi)| \leq \frac{M^{n_N+1}}{(n_N+1)!}\, \prod_{j=0}^{n_N} |b^m \pi - h_j(n_N)|.
\]
Since $h_j(n_N) \in [-1, 1]$ and $b^m \pi \geq \pi > 1$, we have $|b^m \pi - h_j| \leq b^m \pi + 1 \leq 2\,b^m \pi$ for every~$j$.  Hence
\[
    \prod_{j=0}^{n_N} |b^m \pi - h_j(n_N)| \leq (2\,b^m \pi)^{n_N+1} \leq (2\,b^N \pi)^{n_N+1}
\]
for all $m\le N$.  Substituting back:
\begin{equation}\label{eq:term_II_bound}
    |W_N(x) - \mathcal{W}_{N, n_N}(x)|
    \leq \sum_{m=0}^N a^m \cdot \frac{(2\,M\,b^N \pi)^{n_N+1}}{(n_N+1)!}
    \leq \frac{1}{1-a}\cdot\frac{(2\,M\,b^N \pi)^{n_N+1}}{(n_N+1)!}.
\end{equation}

\medskip
Set $R_N := 2\,M\,b^N\pi$.  By the growth condition~\eqref{eq:growth_condition}, we have $R_N/n_N \to 0$ as $N\to\infty$.  A standard consequence of Stirling's approximation is that if $R/n\to 0$, then
\[
   \frac{R^{n+1}}{(n+1)!} \leq \left(\frac{eR}{n+1}\right)^{n+1} \;\longrightarrow\; 0
   \quad\text{as }n\to\infty.
\]

More precisely, for $N$ large enough that $eR_N/(n_N+1) < \frac{1}{2}$ we have
\[
   \frac{R_N^{n_N+1}}{(n_N+1)!}
   \leq \left(\frac{eR_N}{n_N+1}\right)^{n_N+1}
   \leq \left(\frac{1}{2}\right)^{n_N+1}
   \;\longrightarrow\; 0
\]
as $N \to \infty$, giving super-exponential decay.
\medskip
Finally, combining~\eqref{eq:total_error}, \eqref{eq:tail}, and~\eqref{eq:term_II_bound}:
\begin{equation}\label{eq:final_combined}
    \sup_{|x|\le M}|W(x) - \mathcal{W}_{N, n_N}(x)|
    \leq \frac{a^{N+1}}{1-a} + \frac{1}{1-a}\cdot\frac{(2Mb^N\pi)^{n_N+1}}{(n_N+1)!}.
\end{equation}
Both terms converge to~$0$ as $N\to\infty$: the first by $0<a<1$, the second by the analysis above.  Since the bound is independent of~$x$, the convergence is uniform on $[-M,M]$.
\end{proof}

\section{Concluding remarks: Interpretation of the Divergence Wall}

The results of this section are summarized by the following conceptual picture.

The term Divergence Wall introduced in \cite{FRACTALS}  refers to the critical threshold
in the $(N,n)$-plane beyond which the superoscillatory coefficients grow
so rapidly that they overwhelm the geometric decay~$a^m$ of the Weierstrass series.

Fixed-$n$ regime (below the wall): If $n$ is held constant while $N \to \infty$, the Lagrange coefficients for the $m=N$ term grow as $(b^N)^n$, so the general term behaves as $(ab^n)^N$.  Since $ab>1$ implies $ab^n>1$, the series diverges for $x\neq 0$.

The growth condition (above the wall):
    The condition $R_N/n_N = 2Mb^N\pi/n_N \to 0$ ensures that $n_N$ grows faster than $b^N$, so that the factorial growth of $(n_N+1)!$ in the interpolation remainder dominates the exponential growth of the superoscillatory amplitudes.

Fractal reconstruction:
    The joint convergence theorem shows that while the Weierstrass function is nowhere differentiable (and hence requires arbitrarily high frequencies), it can nevertheless be viewed as a uniform limit of smooth, band-limited superoscillating functions---provided the bandwidth parameter~$n$ increases exponentially to outpace the Divergence Wall.

\medskip\medskip
{\bf Conflict of Interest}. The authors declare that they have no competing interests regarding the publication of this paper.

\medskip
{\bf Author contributions}. All authors contributed equally to the study, read and approved the final version of the submitted manuscript.

\medskip
{\bf Availability of data}. There are no data associated with the research in this paper.


\begin{thebibliography}{99}

\bibitem{AAV} Y. Aharonov, D.\,Z. Albert, and L. Vaidman, How the result of a measurement of a component of the spin of a Spin-$\frac{1}{2}$ particle can turn out to be 100, \textit{Phys. Rev. Lett.} \textbf{60} (1988), 1351--1354.

\bibitem{FRACTALS}
F. Colombo, I. Sabadini, D.\,C. Struppa,
On the approximation of Weierstrass function via superoscillations.
preprint. arXiv:2603.05580.

\bibitem{ACSSST21} Y. Aharonov, F. Colombo, I. Sabadini, T. Shushi, D.\,C. Struppa, and J. Tollaksen, A new method to generate superoscillating functions and supershifts, \textit{Proc. Roy. Soc. A} \textbf{477} (2021), 20210020.

\bibitem{colombo2017}
Y. Aharonov, F. Colombo, I. Sabadini, D.\,C. Struppa, and J. Tollaksen,
The Mathematics of Superoscillations,
\textit{Memoirs of the American Mathematical Society}, Vol. 247, No. 1174, 2017.

\bibitem{BOREL} J. Behrndt, F. Colombo, P. Schlosser, and D.\,C. Struppa, Integral representation of superoscillations via complex Borel measures and their convergence, \textit{Trans. Amer. Math. Soc.} \textbf{376} (2023), no.~9, 6315--6340.

\bibitem{berry1994}
M.\,V. Berry,
Faster than Fourier,
in \textit{Quantum Coherence and Reality; in celebration of the 60th Birthday of Yakir Aharonov} (J.\,S. Anandan and J.\,L. Safko, eds.),
World Scientific, Singapore, 1994, pp.~55--65.

\bibitem{berry1980}
M.\,V. Berry and Z.\,V. Lewis, On the Weierstrass-Mandelbrot fractal function, \textit{Proc. R. Soc. Lond. A} \textbf{370} (1980), 459--484.

\bibitem{berry2016fractals}
M.\,V. Berry and S. Morley-Short,
Representing fractals by superoscillations,
\textit{J. Phys. A: Math. Theor.} \textbf{49} (2016), 065203.

\bibitem{berry2019roadmap}
M.\,V. Berry et al.,
Roadmap on superoscillations,
\textit{J. Opt.} \textbf{21} (2019), 053002 (35 pp).

\bibitem{kempf2004}
L. Chojnacki and A. Kempf, New methods for creating superoscillations, \textit{J. Phys. A: Math. Theor.} \textbf{49} (2016), 505203.

\bibitem{Ahar-bohm-JMP}
F. Colombo, E. Pozzi, I. Sabadini, and B.\,D. Wick,
Aharonov--Bohm effect and superoscillations, \textit{J. Math. Phys.} \textbf{66} (2025), no.~7, Paper No. 073501, 16~pp.

\bibitem{jordan} A.\,N. Jordan, Superresolution using supergrowth and intensity contrast imaging, \textit{Quantum Stud.: Math. Found.} \textbf{7} (2020), no.~3, 285--292.

\bibitem{jordan2} A.\,N. Jordan, Y. Aharonov, D.\,C. Struppa, F. Colombo, I. Sabadini, T. Shushi,
J. Tollaksen, J.\,C. Howell, and A.\,N. Vamivakas,
Super-phenomena in arbitrary quantum observables,
\textit{Phys. Rev. A} \textbf{110} (2024), 012206.

\bibitem{kempf4}
A. Kempf,
Four aspects of superoscillations,
\textit{Quantum Stud. Math. Found.} \textbf{5} (2018), 477--484.

\bibitem{Lee} D.\,G. Lee and P.\,J.\,S.\,G. Ferreira, Direct construction of superoscillations, \textit{IEEE Trans. Signal Process.} \textbf{62} (2014), 3125--3134.

\bibitem{mandelbrot1982}
B.\,B. Mandelbrot, \textit{The Fractal Geometry of Nature}, Freeman, San Francisco, CA, 1982.

\bibitem{Pozzi}
E. Pozzi and B.\,D. Wick, Persistence of superoscillations under the Schr\"odinger equation,
\textit{Evol. Equ. Control Theory} \textbf{11} (2022), no.~3, 869--894.

\bibitem{weierstrass} K. Weierstrass, \"Uber continuirliche Functionen eines reellen Arguments, die f\"ur keinen Werth des letzteren einen bestimmten Differentialquotienten besitzen, in: \textit{Mathematische Werke von Karl Weierstrass}, vol.~2, Mayer \& M\"uller, Berlin, 1895, pp.~71--74.

\end{thebibliography}
\end{document}